\documentclass{amsart}

\pdfoutput=1

\usepackage[utf8]{inputenc}
\usepackage{lmodern}
\usepackage[T1]{fontenc}
\usepackage{amssymb}
\usepackage{amsthm}
\usepackage{mathtools}
\allowdisplaybreaks
\usepackage[dvipsnames,svgnames,table]{xcolor}
\usepackage{embrac} 
\ChangeEmph{(}[-.01em,.04em]{)}[.04em,-.05em]
\usepackage{microtype}
\usepackage{textcase}
\usepackage{enumitem}
\usepackage[acronym]{glossaries}
\usepackage{graphicx}
\usepackage{float}
\usepackage{booktabs}
\usepackage{csquotes}
\usepackage{multirow}
\usepackage{tikz}
\usepackage{caption}
\usepackage{subcaption}
\usetikzlibrary{angles,quotes,calc,positioning,shapes,matrix,arrows}
\usepackage{tikz-cd}
\usetikzlibrary{}
\usepackage{tkz-euclide}
\usepackage{etoolbox}
\usepackage{url}
\usepackage[hyperindex=true, colorlinks=true, linkcolor=Maroon, citecolor=DarkGreen, urlcolor=NavyBlue, breaklinks=true,linktocpage=true,linktoc=all]{hyperref}
\usepackage[english,capitalise,noabbrev]{cleveref}
\crefdefaultlabelformat{#2\textup{#1}#3}
\crefname{equation}{equation}{equations}
\usepackage[logical-quotes,backrefs,initials]{amsrefs}

\theoremstyle{plain}
\newtheorem{theorem}{Theorem}[section]

\newtheorem{corollary}[theorem]{Corollary}

\newtheorem{lemma}[theorem]{Lemma}
\newtheorem{proposition}[theorem]{Proposition}

\theoremstyle{definition}

\newtheorem{definition}[theorem]{Definition}

\theoremstyle{remark}

\newtheorem{remark}[theorem]{Remark}

\newcommand{\R}{\mathbb{R}}
\newcommand{\C}{\mathbb{C}}

\renewcommand{\S}{\mathbb{S}}

\newcommand{\RP}{\mathbb{RP}}
\newcommand{\CP}{\mathbb{CP}}
\newcommand{\HP}{\mathbb{HP}}
\newcommand{\OP}{\mathbb{OP}}

\newcommand{\calX}{\mathcal{X}}
\newcommand{\dd}{\mathop{}\!\mathrm{d}}

\newcommand{\contained}{\subset}

\newcommand{\union}{\cup}
\newcommand{\bigunion}{\bigcup}

\newcommand{\suchthat}{\, : \,}
\newcommand{\st}{\suchthat}
\newcommand{\defined}{\coloneqq}

\newcommand{\from}{\colon}

\newcommand{\diam}{\operatorname{diam}}
\newcommand{\vol}{\operatorname{vol}}
\newcommand{\dist}{\operatorname{dist}}

\newcommand{\distributed}{\sim}

\DeclarePairedDelimiterX\norm[1]\lVert\rVert{\ifblank{#1}{\:\cdot\:}{#1}}
\DeclarePairedDelimiterX\abs[1]\lvert\rvert{\ifblank{#1}{\:\cdot\:}{#1}}
\DeclarePairedDelimiterX\set[1]{\{}{\}}{\ifblank{#1}{\: \:}{#1}}
\DeclarePairedDelimiterX\innerprod[2]\langle\rangle
{\ifblank{#1#2}{\,\cdot,\cdot\,}
	{#1,#2}}
\DeclarePairedDelimiterX\floor[1]\lfloor\rfloor{\ifblank{#1}{\:\cdot\:}{#1}}
\DeclarePairedDelimiterXPP\expectation[2]{\ifblank{#1}{\mathbb{E}}{\mathbb{E}_{#1}}}\lparen\rparen{}{\ifblank{#2}{\:\cdot\:}{#2}}
\DeclarePairedDelimiterXPP\Pochhammer[2]{}{(}{)}{_{#2}}{\ifblank{#1}{\cdot}{#1}}
\DeclarePairedDelimiterXPP\card[1]{\#}{(}{)}{}{\ifblank{#1}{\cdot}{#1}}

\newcommand{\Jacobi}[3]{P_{#1}^{(#2,#3)}}

\newcommand{\Discrepancy}{\mathbb{D}}
\newcommand{\Probability}{\mathbb{P}}
\newcommand{\Variance}{\operatorname{Var}}

\newacronym[longplural=\emph{determinantal point processes}]{dpp}{\textup{\textsc{dpp}}}{\emph{determinantal point process}}

\begin{document}

\title[Discrepancy of DPPs on two-point homogeneous spaces]{Discrepancy of determinantal point processes on compact, connected two-point homogeneous spaces}

\author[C. Beltrán]{Carlos Beltrán}
\address{Carlos Beltrán: Departamento de Matemáticas, Estadística y Computación, Universidad de Cantabria,  Avda. Los Castros, s/n, 39005 Santander, Spain}
\email{beltranc@unican.es}

\author[U. Etayo]{Ujué Etayo}
\address{Ujué Etayo: Departamento de Matemáticas, CUNEF Universidad,  C/\,Almansa, 101, 28040 Madrid, Spain}
\email{ujue.etayo@cunef.edu}

\author[G. Gigante]{Giacomo Gigante}
\address{Giacomo Gigante: Dipartimento di Ingegneria Gestionale, dell'Informazione e della Produzione, Università degli Studi di Bergamo, Viale Marconi 5, Dalmine BG, Italy}
\email{giacomo.gigante@unibg.it}

\author[P.\,R. López-Gómez]{Pedro R. López-Gómez}
\address{Pedro R. López-Gómez: Departamento de Matemáticas, CUNEF Universidad,  C/\,Almansa, 101, 28040 Madrid, Spain}
\email{pedro.lopezgomez@cunef.edu}

\author[R.\,W. Matzke]{Ryan W. Matzke}
\address{Ryan W. Matzke: Department of Mathematics, Applied Mathematics and Statistics, Case Western Reserve University, 10900 Euclid Ave., Cleveland, OH 44106, USA}
\email{rwm106@case.edu}

\date{\today{}}

\thanks{Ujué Etayo has been supported by the starting grant from FBBVA associated with the prize José Luis Rubio de Francia and by the Ramón y Cajal Programme of the Spanish Ministry of Science, Innovation and Universities, through the Agencia Estatal de Investigación (AEI), co-funded by the European Social Fund Plus (ESF+), under grant RYC2024-049105-I.\\
Giacomo Gigante is a member of Gruppo Nazionale per l’Analisi Matematica, la Probabilità e le loro
Applicazioni (GNAMPA) of Istituto Nazionale di Alta Matematica (INdAM). He is also very grateful for the kind hospitality provided during his visit to the University of Cantabria by grant PID2020-113887GB-I00 funded by MCIN/AEI/10.13039/501100011033.\\
Ryan W. Matzke was supported by the NSF Postdoctoral Research Fellowship Grant 2202877.\\
Carlos Beltrán, Ujué Etayo,  and Pedro R. López-Gómez were supported by grant PID2020-113887GB-I00 funded by MCIN/AEI/10.13039/501100011033.}

	

\begin{abstract}
    We study the $L^{\infty}$ discrepancy of point sets generated by determinantal point processes on all compact, connected two-point homogeneous spaces, namely spheres and projective spaces. Using concentration inequalities and variance estimates for the number of points in metric balls, we derive general upper bounds for the discrepancy of homogeneous determinantal point processes. In the particular case of the harmonic ensemble, we show that the discrepancy of $N$ points is $O((N^{1-1/D})^{1/2}\log N)$ with high probability, where $D$ denotes the real dimension of the manifold. For the projective ensemble on $\CP^d$, we obtain the sharper bound $O((N^{1-1/D}\log N)^{1/2})$. These results extend previously known discrepancy estimates for determinantal point processes on the sphere to all compact, connected two-point homogeneous spaces.
\end{abstract}

\maketitle

\section{Introduction and main results}

\subsection{Discrepancy}

The $L^\infty$ discrepancy, or simply \emph{discrepancy}, of a collection of points $x_1,\dotsc,x_N$ in a metric measure space $X$ is defined as
\[
\Discrepancy(x_1,\dotsc,x_N)\defined \sup_{A\in \mathcal A}\abs[\bigg]{\card{\set{j\st x_j\in A}}-\frac{N\vol(A)}{\vol(X)}},
\]
where $\mathcal A$ is some predefined family of subsets of $X$, typically chosen to be metric balls or some other class of sets with a simple description. Hence, the discrepancy is a measure of how well distributed the points are in the space: low discrepancy implies that the empirical measure associated with $x_1,\dotsc,x_N$ accurately approximates the volume of the sets in $\mathcal A$. In this paper, $\mathcal A$ will always be the set of all metric balls in $X$.

Finding collections of points with low discrepancy is a classical and usually very difficult problem. Except in a few cases (for example, when $X$ is the interval or the unit square), the most effective known approach is to consider randomized collections of points for increasing values of $N$ and compute the probability that the resulting discrepancy is small. This technique has been successfully applied in many settings, for example when the family $\mathcal A$ consists of rotations, translations, and dilations of a given convex subset of the flat torus \cite[Theorem 18C]{BeckChen}, tilted boxes in the unit cube \cite{beck}, spherical caps in the $D$-sphere  \cite[Theorem 24D]{BeckChen,BeltranMarzoOrtega}, disc segments \cite[Theorem 23B]{BeckChen}, and more generally when $\mathcal A$ is a class of sets with finite Vapnik--Chervonenkis dimension in metric spaces \cite[Theorem 8.5 and Corollary 8.6]{BCCGT}. In this paper, we focus on the case where $X$ is a compact, connected two-point homogeneous space, that is, a sphere $\S^d$; a real, complex, or quaternionic projective space $\RP^d$, $\CP^d$, or $\HP^d$; or the Cayley plane $\OP^2$. For background on these spaces and the harmonic analysis on them, see, for example, \cite[Section 2]{andersonetal}. 

Let us briefly recall some classical and more recent results on discrepancy with respect to metric balls in compact, connected two-point homogeneous spaces.
We should certainly mention first of all J. Beck's result \cite[Theorem 24C]{BeckChen}. For any collection of $N$ points on the $D$-dimensional sphere, there exists a spherical cap $A$ with discrepancy
\[
\abs[\bigg]{\card{\set{j\st x_j\in A}}-\frac{N\vol(A)}{\vol(X)}}>K (N^{1-1/D})^{1/2}.
\]
The constant $K$ depends only on the dimension $D$.
This result was recently extended to all compact, connected two-point homogeneous spaces by M. Skriganov \cite[Theorem 2.2]{Skriganov}. It should be noted that both Beck's and Skriganov's theorems establish a stronger result, since the lower bound is in fact for the $L^2$ average of the quantity $\abs[\big]{\card{\set{j\st x_j\in A}}-N\vol(A)/\vol(X)}$ with respect to a suitable measure $\eta$ on the collection $\mathcal A$:
\begin{align*}
    \Discrepancy_2(x_1,\dotsc,x_N)&\defined \biggl(\int_{\mathcal A}\abs[\bigg]{\card{\set{j\st x_j\in A}}-\frac{N\vol(A)}{{\vol(X)}}}^2 \dd\eta(A)\biggr)^{1/2}\\
    &>K (N^{1-1/D})^{1/2}.
\end{align*}
Concerning the $L^2$ discrepancy, it was also shown in \cite{Skriganov}
that there exist point configurations for which
\[
\Discrepancy_2(x_1,\ldots,x_N)<\hat K (N^{1-1/D})^{1/2}.
\]
The existence of such configurations is proved by probabilistic methods, as usual, but Skriganov also shows that optimal $t$-designs (that is, $t$-designs with an asymptotically minimal number of points) attain the same optimal bound for the $L^2$ discrepancy. In fact, probabilistic methods yield the upper bound $(N^{1-1/D})^{1/2}$ in very general settings even for the $L^p$ discrepancy, for all $p\in[1,+\infty)$ \cite[Corollaries 8.2, 8.3 and 8.4]{BCCGT}. A standard argument that uses the finiteness of the Vapnik--Chervonenkis dimension of the collection $\mathcal A$ allows one to extend the upper estimate to the case $p=+\infty$, at the cost of a logarithmic factor \cite[Theorem 8.5 and Corollary 8.6]{BCCGT}:
\begin{equation*}
    \Discrepancy(x_1,\ldots,x_N)<K (N^{1-1/D}\log(N))^{1/2}.
\end{equation*}
All the examples mentioned at the beginning of this introduction exhibit in fact the same logarithmic loss. It is not known whether this logarithmic term can be removed from the upper bounds.

The random point configurations yielding the quasi-optimal discrepancy estimates in the cases mentioned above are essentially obtained by selecting one point uniformly in each element of an area regular partition of $X$. A more general approach is to estimate the discrepancy of a random collection of points drawn from a \gls{dpp}. This has been done in \cite{alishashi-zamani} for the spherical ensemble on the $2$-dimensional sphere, matching the bound $(N^{1/2}\log N)^{1/2}$, and in \cite{BeltranMarzoOrtega} for the harmonic ensemble on the $D$-dimensional sphere, yielding the slightly worse bound $(N^{1-1/D})^{1/2}\log N$. In \cite{BordaGrabnerMatzke}, the authors studied the expected $L^2$ discrepancy of the projective ensemble on the complex projective spaces $\mathbb{CP}^d$, and the harmonic ensemble for spheres $\mathbb{S}^d$, finding upper bounds $(N^{1-1/2d})^{1/2}$ and $(N^{1-1/2d})^{1/2} \sqrt{\log N}$, respectively. The first bound is optimal, but the second has an extra factor of $\sqrt{\log N}$, similar to the findings for the $L^{\infty}$ discrepancy of the harmonic ensemble in \cref{cor:Discrepancy-harmonic} below.

The main goal of this paper is to extend the techniques used to study the discrepancy of \glspl{dpp} on spheres to the remaining compact, connected two-point homogeneous spaces.

\subsection{Determinantal point processes}

In this section, we introduce the basic concepts on \glspl{dpp} that we will use in this work. For a more comprehensive reference, we refer the reader to \cite[Chapter 4]{GAF}.

Let $\Lambda$ be a locally compact, Polish topological space endowed with a Radon measure $\mu$ (in particular, any compact Riemannian manifold endowed with its natural measure satisfies these hypotheses). A \emph{simple point process} $\calX$ of $N$ points on $\Lambda$ is a random variable taking values in the space of $N$-point subsets of $\Lambda$. In other words, a simple point process selects $N$ points at random in $\Lambda$, and with probability $1$ they are pairwise different.
q
The \emph{joint intensities}, or \emph{correlation functions}, are functions $\rho_{k}\from \Lambda^k\to[0,\infty)$, with $k\geq1$, such that for any family of mutually disjoint subsets $A_1,\dotsc,A_k$ of $\Lambda$ we have
\begin{equation*}
  \expectation[\Big]{x\sim\mathcal{X}}{\prod_{i=1}^k \card{x\cap A_i}}
  =
  \int_{\prod A_i}\rho_k({y_1,\dotsc,y_k})\dd\mu({y_1})\dotsb\dd\mu({y_k}).
\end{equation*} 
By $x\sim\mathcal{X}$ we mean that $x=\set{x_1,\dotsc,x_N}$ is a subset of $N$ elements of $\Lambda$ sampled from the point process $\calX$.

From \cite[Formula (1.2.2)]{GAF}, for any measurable function $\phi\from \Lambda^{k} \to [0, \infty)$ the following equality holds:
\begin{equation}\label{eq:joint-intensities}
	\expectation[\Big]{x\distributed\mathcal{X}}{\sum_{\substack{i_1,\dotsc,i_k\\ \text{distinct}}}\!\!\phi(x_{i_1},\dotsc,x_{i_k})}=\int\limits_{\Lambda^k}\phi(y_1,\dotsc,y_k)\rho_k(y_1,\dotsc,y_k)\dd\mu(y_1)\dotsb\dd\mu(y_k).
\end{equation}
If there exists a measurable function $K\from \Lambda\times\Lambda\to\C$ such that these joint intensities can be written as 
\begin{equation*}
    \rho_{k}(x_{1},\dotsc,x_{k}) = \det(K(x_{i},x_{j}))_{1\leq i,j\leq k},
\end{equation*}
then we say that $\calX$ is a \emph{determinantal point process} with kernel $K$. A particularly suitable collection of such processes is obtained by choosing $K$ as the reproducing kernel of an $N$-dimensional subspace $H$ of the Hilbert space $L^2(\Lambda,\C)$. Recall that the reproducing kernel of $H$ is the unique continuous, Hermitian, positive-definite function $K_H\from \Lambda\times\Lambda\to\C$ such that $K_H(\cdot,x)\in H$ for all $x\in\Lambda$ and
\begin{equation*}
    f(x)=\innerprod{f}{K_H(\cdot,x)}=\int_{\Lambda}f(y)K_H(x,y)\dd \mu(y),\qquad \forall x\in\Lambda,\ \forall f\in H.
\end{equation*}
Given any orthonormal basis $\varphi_{1},\dotsc,\varphi_{N}$ of $H$, we have
\begin{equation*}
     K_H(x,y)=\sum_{i=1}^N\varphi_i(x)\overline{\varphi_i(y)},
\end{equation*}
and we say that $K_H$ is a \emph{projection kernel of trace $N$}.

The following result, which is a direct consequence of the Macchi--Soshnikov theorem (see \cite[Theorem 4.5.5]{GAF}), guarantees the existence of determinantal point processes associated with projection kernels.
 
\begin{proposition}\label{prop:expectation-dpp}	
Let $\Lambda$ be a locally compact, Polish topological space with a Radon measure $\mu$ and let $H \subset L^2(\Lambda,\mathbb C)$ have dimension $N$. Let $K_H$ be the reproducing kernel of $H$. Then, there exists a simple point process $\mathcal{X}_H$ in $\Lambda$ of $N$ points with associated joint intensity functions
\begin{equation*}
	\rho_{k}(x_{1},\dotsc,x_{k}) = \det(K_H(x_{i}, x_{j}))_{1\leq i,j\leq k}.
\end{equation*}
\end{proposition}
\begin{remark}\label{remark:KH-xx}

Under the hypotheses of Proposition \ref{prop:expectation-dpp}, from \eqref{eq:joint-intensities} with $\phi\equiv 1$ and $k=1$ we have
\begin{equation}\label{eq:expected-N-dpp}
    N=\expectation{x\sim\mathcal{X}_H}{N} 
    = \int_{\Lambda} K_H(x,x)\dd\mu(x).
\end{equation}
\end{remark}
We say that a \gls{dpp} $\mathcal X$ in a compact, connected two-point homogeneous space $\Omega$ is \emph{homogeneous} if it is invariant under the action of the isometry group of $\Omega$. More specifically, if we denote by $G$ the isometry group of $\Omega$, this means that
\begin{equation*}
    K(gx,gy)=K(x,y),\qquad \forall g\in G,\quad \forall x,y\in\Omega.
\end{equation*}
Since the space is homogeneous, the isometry group acts transitively and the previous equation implies that $K(x,x)$ is constant.

Following the notation in \cite{andersonetal}, when $\Lambda=\Omega$ is a compact, connected two-point homogeneous space, we will use the normalization $\vol(\Omega)=1$. Then, if the \gls{dpp} $\mathcal{X}_H$ is homogeneous, since $K_H(x,x)$ is constant we deduce from \eqref{eq:expected-N-dpp} that 
\begin{equation*}
    K_H(x,x)=N.
\end{equation*}
In addition, it is easy to see that the expected value of the number of points that fall inside a measurable set $A\contained \Omega$ is precisely $N\vol(A)$.

In this work, we focus on two specific homogeneous \glspl{dpp}: the harmonic ensemble and the projective ensemble, the latter only for the complex projective space $\CP^d$.

\subsubsection{The harmonic ensemble on compact, connected two-point homogeneous spaces}\label{sec:harmonic-ensemble}

The harmonic ensemble on a compact Riemannian manifold $\mathcal{M}$ is the \gls{dpp} generated by the reproducing kernel of the finite-dimensional subspace of $L^2(\mathcal{M})$ spanned by the eigenfunctions of the Laplace--Beltrami operator on $\mathcal{M}$ corresponding to eigenvalues up to a certain threshold. It was first introduced in the case of the sphere $\S^d$ in \cite{BeltranMarzoOrtega} and later generalized to all the compact, connected two-point homogeneous spaces in \cite{andersonetal}. 

In what follows, we denote any such compact, connected two-point homogeneous space by $\Omega$. For our purposes, it suffices to note that the kernel of the \gls{dpp} known as the harmonic ensemble on $\Omega$ is given, for $L\geq 1$, by
\begin{equation}\label{eq:kernel-harmonic}
    K_L^{(\alpha,\beta)}(p,q)
    =
    \frac{\Pochhammer{\alpha+\beta+2}{L}}{\Pochhammer{\beta+1}{L}}\Jacobi{L}{\alpha+1}{\beta}\bigl(\cos(2\kappa\dist(p,q))\bigr),
\end{equation}
where $\dist(p,q)$ is the Riemannian distance in $\Omega$ from $p$ to $q$, $\Jacobi{L}{\alpha}{\beta}$ is the Jacobi polynomial of degree $L$ with parameters $\alpha$ and $\beta$,  $\Pochhammer{a}{L}$ is the Pochhammer symbol or rising factorial, and $\kappa=1/2$ for the sphere and $\kappa=1$ for the projective spaces. The number of points sampled by this \gls{dpp} is given by
\begin{equation*}
    N=\pi_L=\frac{\Pochhammer{\alpha+\beta+2}{L}\Pochhammer{\alpha+2}{L}}{\Pochhammer{\beta+1}{L}L!} \sim \frac{\Gamma(\beta +1)}{\Gamma(\alpha + \beta + 2)\Gamma(\alpha +2)}L^{2\alpha + 2}.
\end{equation*}
The values of $\alpha$ and $\beta$ are fixed for each $\Omega$; see Table \ref{table:alphaomega}.

\setlength{\tabcolsep}{10pt}
\begin{table}[h!]
\centering
\caption{Values of $\alpha$, $\beta$, and the real dimension $D=2\alpha+2$ for each two-point homogeneous space.}
\label{table:alphaomega}
\begin{tabular}{lrrr}
\toprule
$\Omega$            & $\alpha$      & $\beta$   & $D$ \\ 
\midrule
$\S^{d}$          & $(d-2)/2$     & $(d-2)/2$ & $d$ \\ 
$\RP^{d}$         & $(d-2)/2$     & $-1/2$    & $d$ \\ 
$\CP^{d}$         & $d-1$         & $0$       & $2d$ \\ 
$\HP^{d}$         & $2d-1$        & $1$       & $4d$ \\ 
$\OP^{2}$           & $7$           & $3$       & $16$ \\
\bottomrule
\end{tabular}
\end{table}

\subsubsection{The projective ensemble}\label{sec:projective-ensemble}

In \cite{BeltranEtayo}, the authors introduced an extension to $\CP^d$ of the spherical ensemble studied in \cite{alishashi-zamani} for the two-dimensional sphere (which is diffeomorphic to the complex projective line $\CP^1$). They called this \gls{dpp} the \emph{projective ensemble}. The kernel of this \gls{dpp} satisfies, for $L\geq 1$,
\begin{equation}\label{eq:kernel-projective}
    \abs[\big]{K_\star^{(N,d)}(p,q)}=N\abs[\bigg]{\innerprod[\bigg]{\frac{p}{\norm{p}}}{\frac{q}{\norm{q}}}}^L,
\end{equation}
where $N$ is the number of points sampled by this random process. Here, $p$ and $q$ are affine representatives of the corresponding projective points in $\CP^d$. The parameters $N$ and $L$ are related by the following formula:
\begin{equation}\label{eq:dyL}
    N=\binom{d+L}{d}\sim L^d.
\end{equation}
Hence, both \glspl{dpp} can only generate $N$ points for $N$ in an infinite subsequence of the natural numbers, but not for arbitrary integers $N$.

\subsection{Main results}

Our first main result is the following theorem, which provides the fundamental tool that enables us to establish the bounds for the discrepancy of the two homogeneous \glspl{dpp} analyzed in this work. Here, $N_A$ denotes the number of points coming from the random process that fall inside $A$.

\begin{theorem}\label{th:maintool}
    Let $\mathcal X$ be a homogeneous \gls{dpp} on $\Omega$, and let $M>0$. Then, with probability at least $1-N^{-M}$, the discrepancy of $N$ points sampled from the \gls{dpp} on $\Omega$ satisfies
    \begin{equation*}
        \Discrepancy(x_1,\ldots,x_N)=O\Bigl(\log N+\bigl(\log N\sup_{A\in\mathcal A}\Variance(N_A)\bigr)^{1/2}\Bigr),
    \end{equation*}
    where $\mathcal A$ denotes the collection of all metric balls in $\Omega$.
\end{theorem}

With \cref{th:maintool} in hand, deriving the desired bounds on the discrepancy reduces to computing the variance of $N_A$ for the \glspl{dpp} under consideration. For the harmonic ensemble, we obtain the following result.

\begin{theorem}\label{th:Variance-harmonic}
    Let $\Omega$ be a compact, connected two-point homogeneous space with real dimension $D$. Let $A$ be a metric ball in $\Omega$, and let $N_A$ denote the number of points from the harmonic ensemble that fall inside $A$. Then,
    \begin{equation*}
        \Variance(N_A)\leq KN^{1-1/D}\log N,
    \end{equation*}
    for some constant $K>0$ independent of $A$.
\end{theorem}

Then, combining \cref{th:maintool,th:Variance-harmonic}, we immediately obtain the following corollary.

\begin{corollary}\label{cor:Discrepancy-harmonic}
    For every $M>0$, the discrepancy of $N=\pi_L$ points $x_1,\dotsc,x_N\in \Omega$ drawn from the harmonic ensemble satisfies
    \begin{equation*}
        \Discrepancy(x_1,\ldots,x_N)= O\biggl(\bigl(N^{1-1/D}\bigr)^{1/2}\log N\biggr)
    \end{equation*}
    with probability at least $1-N^{-M}$.
\end{corollary}

In the case of the projective ensemble introduced in \cref{sec:projective-ensemble}, we prove the following result, which shows that the projective ensemble improves the bound for the harmonic ensemble by a factor of $\log N$.

\begin{theorem}\label{th:Variancep}
    Let $A$ be a metric ball in $\CP^d$, let $D=2d$ denote the real dimension of $\CP^d$, and let $N_A$ be the number of points drawn from the projective ensemble that fall inside $A$. Then,
    \begin{equation*}
        \Variance(N_A)\leq KN^{1-1/D},
    \end{equation*}
    for some constant $K>0$ independent of $A$.
\end{theorem}

As a direct consequence of \cref{th:maintool,th:Variancep}, we obtain the following corollary.

\begin{corollary}\label{cor:Discrepancy-projective}
    For every $M>0$, the discrepancy of $N=\binom{d+L}{d}$ points $x_1,\dotsc,x_N\in \CP^d$ drawn from the projective ensemble satisfies
    \begin{equation*}
        \Discrepancy(x_1,\ldots,x_N)= O\biggl(\bigl(N^{1-1/D}\bigr)^{1/2}\sqrt{\log N}\biggr)
    \end{equation*}
    with probability at least $1-N^{-M}$.
\end{corollary}

The paper is organized as follows. First, in \cref{sec:general} we prove \cref{th:maintool}, our general result. Next, in \cref{sec:variancesh} we compute the variance for the harmonic ensemble and prove \cref{th:Variance-harmonic}. Finally, in \cref{sec:variancesproj} we compute the variance for the projective ensemble and prove \cref{th:Variancep}.

\section{A general approach to bound the discrepancy of homogeneous \glspl{dpp}}\label{sec:general}

In order to prove \cref{th:maintool}, we need the following tools.

\begin{definition}
    An \emph{Ahlfors regular metric measure space}, or simply Ahlfors regular space, of dimension $D>0$ is a complete metric space $X$ with a Borel measure $\mu$ with the property that there are two positive constants $c_1$ and $c_2$ such that all open metric balls $B(x,r)$ with $x\in X$ and $0<r<\diam(X)$ satisfy the bounds
    \begin{equation*}
        c_1r^D\leq \mu(B(x,r))\leq c_2r^D.
    \end{equation*}
\end{definition}

\begin{theorem}\label{th:Aprime}
Let $X$ be an Ahlfors regular compact space of dimension $D$. Then there exists a constant $C>0$
such that for any $n\in\mathbb N$ there exists a finite collection $\mathcal A'$ of metric balls \embparen{including the empty ball $\emptyset$} such that
\begin{equation*}
\card{\mathcal A'}\le C n^{D+1}
\end{equation*}
and such that, for any $x\in X$ and any $r\geq 0$, there exist $A_1=B(x_1,r_1),A_2=B(x_2,r_2)\in \mathcal A'$ with
\begin{equation*}
A_1\subseteq B(x,r)\subseteq A_2
\qquad\text{and}\qquad
r_2-r_1\le \frac1n .
\end{equation*}
\end{theorem}

\begin{proof}
Since $X$ is compact, it has finite total volume. 
Hence, for every $\varepsilon>0$ there exists a maximal $\varepsilon$-separated net $S\subset X$. This means that, on the one hand, $d(x_i,x_j)\geq\varepsilon$ for every two different points in $S$, 
and, on the other hand, if another point is added to $S$, this separation property is lost. The existence of this maximal set is clear: take any point $x_1\in X$, and add any other point $x_2$ which is $\varepsilon$ far from $x_1$, then any other one that is $\varepsilon$ far from both $x_1,x_2$ and continue this process. 
Since the points are $\varepsilon$ separated, the open balls of radius $\varepsilon/2$ centered at points in the sequence are disjoint, and after $n$ steps we have
\[
\vol(X)\geq\vol\Bigl(\bigunion_{i=1}^nB(x_i,\varepsilon/2)\Bigr)=\sum_{i=1}^n \vol(B(x_i,\varepsilon/2))\geq c_1 n \varepsilon^D/2^D,
\]
which means that the process is finite and the family has at most $C\varepsilon^{-D}$ points, with $C$ a constant. Note that in the last inequality we have used the Ahlfors regularity of $X$. Moreover, for any $x\in X$ we have $d(x,S)<\varepsilon$, since otherwise we could add the point to the family and it would not be maximal. Therefore, choosing $\varepsilon=1/(4n)$, we have:
\begin{equation*}
    \card{S}\le C n^D .
\end{equation*}
Let $\diam(X)$ denote the diameter of $X$ and define
\begin{equation*}
    \delta=\frac{1}{2n}=2\varepsilon.
\end{equation*}
Define the discrete set of radii
\begin{equation*}
\mathcal R \defined \set[\bigg]{0,\delta,2\delta,\dotsc,K\delta},
\end{equation*}
where $K$ is the minimal number that satisfies $K\delta \geq \diam(X)+1/n$. Then $\card{\mathcal R}=O(n)$.
Define
\begin{equation*}
\mathcal A' \defined \set{B(s,r) \st s\in S,\ r\in\mathcal R} \union \set{\emptyset}.
\end{equation*}
This family is finite and its cardinality satisfies
\begin{equation*}
    \card{\mathcal A'} \le \card{S}\card{\mathcal R} +1
    \leq C' n^D \cdot O(n)
    = O(n^{D+1}).
\end{equation*}
We have constructed $\mathcal A'$; now, we verify the claim of the theorem. 
Let $x\in X$ and $r\geq 0$ be arbitrary. Choose $s\in S$ such that $d(x,s)<\varepsilon$. If $r>\diam(X)$, then $B(x,r)=X$ and we take $A_1=A_2=B(s,K\delta)$. Otherwise, define
\begin{equation*}
    t\defined\max\set{r-\varepsilon,0},
    \qquad
    r_1 \defined \delta\floor[\bigg]{\frac{t}{\delta}},
    \qquad
    r_2 \defined r_1 + 2\delta .
\end{equation*}
Then $r_1,r_2\in\mathcal R$ and
\begin{equation*}
    r_2-r_1 = 2\delta = \frac1n .
\end{equation*}
If $r_1 = 0$, then take $A_1 = \emptyset$. 
Otherwise, let $y\in B(s,r_1)$. Then
\begin{equation*}
    d(x,y)
    \le d(x,s)+d(s,y)
    < \varepsilon + r_1
    \le \varepsilon + t
    \le r.
\end{equation*}
Hence
\begin{equation*}
    B(s,r_1)\subseteq B(x,r).
\end{equation*}
Now let $y\in B(x,r)$. Then
\begin{equation*}
    d(s,y)
    \le d(s,x)+d(x,y)
    < \varepsilon + r .
\end{equation*}
Since $\floor{u}>u-1$, we have
\begin{equation*}
r_1
=
\delta \floor[\bigg]{\frac{t}{\delta}}
> t-\delta
= r-\varepsilon-\delta .
\end{equation*}
Therefore
\begin{equation*}
r_2=r_1+2\delta> r-\varepsilon+\delta =r+\varepsilon,
\end{equation*}
and thus
\begin{equation*}
B(x,r)\subseteq B(s,r_2).
\end{equation*}
Setting
\begin{equation*}
A_1:=B(s,r_1),
\qquad
A_2:=B(s,r_2),
\end{equation*}
we obtain
\begin{equation*}
A_1\subseteq B(x,r)\subseteq A_2
\qquad\text{and}\qquad
r_2-r_1\le \frac1n .
\end{equation*}
This completes the proof. 
\end{proof}

In order to obtain upper bounds for the discrepancy we use a version of Bernstein's inequality. As above, given a metric ball $A\subseteq \Omega$ and a collection of $N$ points in $\Omega$, we denote by $N_A$ the number of points in the collection that fall inside $A$.

\begin{theorem}[Bernstein--Chernoff bound]\label{th:Bernstein}
Let $\mathcal X$ be a homogeneous \gls{dpp} on $\Omega$, and let $A$ be a metric ball in $\Omega$. Then,
\begin{equation*}
    \Probability\bigl(\abs{N_A - \mathbb{E}(N_A)} \geq t\bigr) \leq    \begin{cases}
        2e^{-t/4}&t\geq \Variance(N_A),\\
        2e^{-t^2/(4\Variance(N_A))}&t\leq \Variance(N_A).
    \end{cases}
\end{equation*}
\end{theorem}

\begin{proof}
    This is a simplified version of \cite[Lemma 2]{beck}. Indeed, from \cite[Theorem 4.5.3 and Remark 4.2.5]{GAF} we have that the restriction of $\mathcal X$ to $A$ is a determinantal point process, and $N_A$ is equal in distribution to a sum of independent Bernoulli random variables $\eta_1,\dotsc,\eta_N$ with parameters $p_1,\dotsc,p_N$ that we do not need to compute. Here, $N$ is the total number of points sampled by the \gls{dpp}. Then,
    \begin{equation*}
    N_A - \mathbb{E}(N_A)=\sum_{i=1}^N(\eta_i-p_i).
    \end{equation*}
    The random variables $\eta_i-p_i$ are independent, have mean zero, and their absolute value is at most $1$. Therefore, we can apply \cite[Lemma 2]{beck}, which gives
    \[
    \Probability\biggl(\abs[\bigg]{\sum_{i=1}^N(\eta_i-p_i)}\geq t\biggr)\leq
    \begin{cases}
        2e^{-t/4}&t\geq\beta^2,\\
        2e^{-t^2/(4\beta^2)}&t\leq\beta^2,
    \end{cases}
    \]
    where $\beta^2$ is the variance of $\sum_{i=1}^N(\eta_i-p_i)$, that is, $\beta^2=\Variance(N_A)$. This completes the proof.
\end{proof}

\begin{proof}[Proof of Theorem \ref{th:maintool}]
Note that $M>0$ is fixed during the proof. Let $\mathcal A'$ denote the collection of at most $cN^c$ balls provided by \cref{th:Aprime} 
(the theorem states that there exists a set of $CN^{D+1}$ balls, it suffices to take $c=\max(C,D+1)$).
For any fixed $A'\in\mathcal A'$, we apply \cref{th:Bernstein} with the following choices for the parameter $t$:
\begin{itemize}
    \item If $\Variance(N_{A'})> 4(M+c)\log N+4\log(4c)$, set
    \begin{equation*}
t_{A'}=2\sqrt{\Variance(N_{A'})}\sqrt{(M+c)\log N+\log(4c)}.
    \end{equation*}
    Note that
    \[
    \frac{t_{A'}}{\Variance(N_{A'})}=\frac{2\sqrt{(M+c)\log N+\log(4c)}}{\sqrt{\Variance(N_{A'})}}<1.
    \]
    From \cref{th:Bernstein} we have
\[
\abs[\big]{N_{A'}-\mathbb E (N_{A'})}\leq t_{A'}= O\Bigl(\bigl(\log N\sup_{A\in\mathcal A}\Variance(N_A)\bigr)^{1/2}\Bigr)
\]
with probability at least
\[
1-2e^{-(M+c)\log N-\log(4c)}=1-\frac{1}{2cN^{M+c}}.
\]

\item Otherwise, set
\begin{equation*}
    t_{A'}=4\bigl((M+c)\log N+\log(4c)\bigr)\geq \Variance(N_{A'}).
\end{equation*}
Again from \cref{th:Bernstein} we have
\[
\abs[\big]{N_{A'}-\mathbb E (N_{A'})}\leq t_{A'}= O(\log N)
\]
with probability at least
\[
1-2e^{-(M+c)\log N-\log(4c)}=1-\frac{1}{2cN^{M+c}}.
\]
\end{itemize}
We have proved that, given any metric ball $A'$ in $\mathcal{A'}$,
\[
\abs[\big]{N_{A'}-\mathbb E (N_{A'})}= O\Bigl(\log N+\bigl(\log N\sup_{A\in\mathcal A}\Variance(N_A)\bigr)^{1/2}\Bigr)
\]
with probability at least $1-\frac{1}{2cN^{M+c}}$.

Since $\mathcal A'$ has at most $cN^c$ elements, we conclude that, with probability at least
\[
1-\frac{cN^c}{2cN^{M+c}}=1-\frac{N^{-M}}{2},
\]
the following inequality holds:
\[
\abs{N_{A'}-\expectation{}{N_{A'}}}\leq t\qquad \forall A'\in\mathcal A',
\]
where
\[
t=O\Bigl(\log N+\bigl(\log N\sup_{A\in\mathcal A}\Variance(N_A)\bigr)^{1/2}\Bigr).
\]
Now, let $A$ be any ball in $\Omega$, and let $r$ be its radius. From the definition of $\mathcal A'$, there exist balls $A'_1,A'_2\in\mathcal A'$ with respective radii $r'_1,r'_2$ such that $A'_1\subseteq A\subseteq A'_2$ and $r'_2-r'_1\leq 1/N$. Then, with probability at least $1-N^{-M}/2$,
\begin{align*}
    N_{A}-\expectation{}{N_{A}}&=N_A-N{\vol(A)}\\
    &\leq N_{A'_2}-N{\vol(A'_1)} \\
    &\leq \left|N_{A'_2}-N{\vol(A'_2)}\right|+N\abs[\bigg]{{\vol(A'_2)}-{\vol(A'_1)}} \\
    &\leq   \abs[\big]{N_{A'_2}-\expectation{}{N_{A'_2}}}+N\abs[\bigg]{{\vol(A'_2)}-{\vol(A'_1)}} \\
    &\leq  t+ CN(r_{2}'^{D}-r_1'^D)\\
    &\leq  t+ CN(r_2'-r_1')\\
    &\leq  t+ C,
\end{align*}
for some constant $C>0$ (which may change from one expression to the next one). Here we have used the Ahlfors regularity of the space in the fourth inequality and the fact that $\sum_{k=0}^{D-1} r_2'^{D-1-k} r_1'^k \leq \text{Constant}$ (since the manifold has finite diameter) in the previous to last step. 

Similarly, with probability at least $1-N^{-M}/2$,
\begin{align*}
    \expectation{}{N_{A}}-N_{A}&=N{\vol(A)}-N_A\\
    &\leq N{\vol(A'_2)} -N_{A'_1}\\
    &\leq  \abs[\bigg]{N_{A'_1}-N{\vol(A'_1)}}
    +N\abs[\bigg]{{\vol(A'_1)}-{\vol(A'_2)}} \\
    &= \abs[\big]{N_{A'_1}-\expectation{}{N_{A'_1}}}+N\abs[\bigg]{{\vol(A'_2)}-{\vol(A'_1)}} \\
    &\leq t+ CN(r_2'^{D}-r_1'^D)\\
    &\leq t+ CN(r_2'-r_1')\\
    &\leq t+ C.
\end{align*}
Combining the above estimates, for any ball $A\subseteq X$ we have
\begin{equation*}
    \abs[\big]{N_{A}-\mathbb E (N_{A})}\leq t+C
\end{equation*}
with probability at least $1-N^{-M}$. That is, the discrepancy is at most $t+C$, as claimed. 
\end{proof}

\section{The variance of the harmonic ensemble on the two-point homogeneous spaces. Proof of \cref{th:Variance-harmonic}}\label{sec:variancesh} 

In this section, we prove \cref{th:Variance-harmonic}. We follow the ideas in \cite[Proof of Proposition 2]{BeltranMarzoOrtega}. By \cite[Formula 28]{Virag}, the variance of $N_A$ is given by
\begin{equation*} 
    \Variance(N_A)=\int_{A}\int_{A^c}\abs[\big]{K_L^{(\alpha,\beta)}(p,q)}^2\dd q \dd p,
\end{equation*}
where $A^c$ is the complement of $A$. Let $a$ and $r$ denote the center and radius of $A$, respectively, with $r\in(0,\pi/(2\kappa))$.

For any $p\in A$, if we set $s=\dist(p,a)$, we have
\begin{equation*}
    A^c\subseteq B(p,r-s)^c.
\end{equation*}
Therefore,
\begin{equation*}
    \Variance(N_A)
    \leq 
    \biggl(
    \frac{\Pochhammer{\alpha+\beta+2}{L}}{\Pochhammer{\beta+1}{L}}
    \biggr)^2
    \int_{p\in A}\int_{q\in B(p,r-s)^c}\abs[\bigg]{\Jacobi{L}{\alpha+1}{\beta}(\cos(2\kappa\dist(p,q))}^{2}\dd q \dd p.
\end{equation*}
The inner integral depends only on the distance from $q$ to $p$, and can therefore be reduced to a univariate integral (see \cite[Eq. (2.8)]{andersonetal} for the weighted measure in each projective space, or \cite[Eq. (5.4)]{BGG} for all cases. Note that this convention is consistent with our normalization $\vol(\Omega)=1$). Then, the inner integral is
\begin{equation*}
    \frac{2\kappa\Gamma(\alpha + \beta + 2)}{\Gamma(\alpha + 1)\Gamma(\beta +1)}
    \int_{r-s}^{\pi/(2\kappa)}
    \abs[\bigg]{\Jacobi{L}{\alpha+1}{\beta}(\cos(2\kappa\theta))}^{2}
    (\sin\kappa\theta)^{2\alpha+1}(\cos\kappa\theta)^{2\beta+1} \dd\theta.
\end{equation*}
Since $s=\dist(p,a)$, the outer integral also depends only on the distance from $p$ to the fixed point $a$, and can likewise be written as a univariate integral. Overall, we have
\begin{multline*}
    \Variance(N_A)
    \leq
    \biggl(
    \frac{\Pochhammer{\alpha+\beta+2}{L}}{\Pochhammer{\beta+1}{L}}
    \frac{2\kappa\Gamma(\alpha + \beta + 2)}{\Gamma(\alpha + 1)\Gamma(\beta +1)}
    \biggr)^2
    \\
    \qquad\times
    \int_{0}^r
    (\sin\kappa\phi)^{2\alpha+1} (\cos\kappa\phi)^{2\beta+1} 
    \int_{r-\phi}^{\pi/(2\kappa)}
    \abs[\bigg]{\Jacobi{L}{\alpha+1}{\beta}(\cos(2\kappa\theta))}^{2}\\
    \times
    (\sin\kappa\theta)^{2\alpha+1}(\cos\kappa\theta)^{2\beta+1} \dd\theta\dd \phi.
\end{multline*}
A quick change of variables gives
\begin{multline}\label{eq:Var}
    \Variance(N_A)
    \leq
    \biggl(
    \frac{\Pochhammer{\alpha+\beta+2}{L}}{\Pochhammer{\beta+1}{L}}
    \frac{2\Gamma(\alpha + \beta + 2)}{\Gamma(\alpha + 1)\Gamma(\beta +1)}
    \biggr)^2
    \\
    \times
    \int_{0}^{\kappa r}
    (\sin\phi)^{2\alpha+1} (\cos\phi)^{2\beta+1} 
    \int_{\kappa r-\phi}^{\pi/2}
    \abs[\bigg]{\Jacobi{L}{\alpha+1}{\beta}(\cos(2\theta))}^{2}\\
    \times
    (\sin\theta)^{2\alpha+1}(\cos\theta)^{2\beta+1} \dd\theta\dd \phi.
\end{multline}
Here,
\begin{equation*}
    \frac{2\Gamma(\alpha + \beta + 2)}{\Gamma(\alpha + 1)\Gamma(\beta +1)}
\end{equation*}
is a constant independent of $L$. Hence, we can neglect it. Also, we have
\begin{equation*}
    \biggl(
    \frac{\Pochhammer{\alpha+\beta+2}{L}}{\Pochhammer{\beta+1}{L}}
    \biggr)^2\sim L^{2\alpha+2}\sim N.
\end{equation*}
Now, we claim that the double integral in \eqref{eq:Var} satisfies
\begin{align*}
    I&=\int_{0}^{\kappa r}
    \int_{\kappa r-\phi}^{\pi/2}
    \abs[\bigg]{\Jacobi{L}{\alpha+1}{\beta}(\cos(2\theta))}^{2}
    (\sin\phi)^{2\alpha+1} (\cos\phi)^{2\beta+1} \\
    &\hspace{6cm}\times
    (\sin\theta)^{2\alpha+1}(\cos\theta)^{2\beta+1} \dd\theta\dd \phi\\
    &\leq KL^{-1}\log L,
\end{align*}
for some constant $K>0$. Assuming this claim, we have
\[
\Variance(N_A)= O(NL^{-1}\log N)=O(N^{1-1/D}\log N),
\]
and \cref{th:Variance-harmonic} follows. Let us then prove the claim on $I$. Using that $\cos \phi\leq 1$ and $\sin\phi\le1$, we have
\begin{equation}\label{eq:integral}
    I\leq \int_{0}^{\kappa r}
    \int_{\kappa r-\phi}^{\pi/2}
    \abs[\bigg]{\Jacobi{L}{\alpha+1}{\beta}(\cos(2\theta))}^{2}
    (\sin\theta)^{2\alpha+1} (\cos\theta)^{2\beta+1}\dd\theta\dd \phi.
\end{equation}
To estimate the double integral in \eqref{eq:integral}, we will use the following bounds for the Jacobi polynomials from 
\cite[Eq. (7.32.5) {and Theorem 8.21.8}  along with Eq.~(4.1.3)]{szego1975orthogonal}:
\begin{equation*}
    \Jacobi{L}{\alpha+1}{\beta}(\cos2\theta)=\begin{cases}
         O(L^{\alpha+1}) & 0\leq \theta\leq L^{-1},\\
        (\sin\theta)^{-\alpha-3/2}(\cos\theta)^{-\beta-1/2}O(L^{-1/2}) & L^{-1}\leq \theta\leq \pi/2-L^{-1},\\
        O(L^{\beta}) & \pi/2-L^{-1}\leq \theta\leq \pi/2.
    \end{cases}
\end{equation*}
Now, consider the subdivision of the region of integration shown in \cref{fig:subdivision}, where
\begin{align*}
    R_1&=\set{(\phi,\theta)\st 0\leq \phi\leq \kappa r,\ {\max(\pi/2-L^{-1},\kappa r-\phi)}\leq \theta\leq \pi/2},\\
    R_2&=\set{(\phi,\theta)\st {\max(0,\kappa r-\pi/2+L^{-1})}\leq \phi\leq \kappa r-L^{-1},\ \kappa r-\phi\leq \theta\leq \pi/2-L^{-1}},\\
    R_3&=\set{(\phi,\theta)\st \kappa r-L^{-1}\leq \phi\leq \kappa r,\ L^{-1}\leq \theta\leq \pi/2-L^{-1}},\\
    R_4&=\set{(\phi,\theta)\st \kappa r-L^{-1}\leq \phi\leq \kappa r,\ \kappa r-\phi\leq \theta\leq L^{-1}}.
\end{align*}

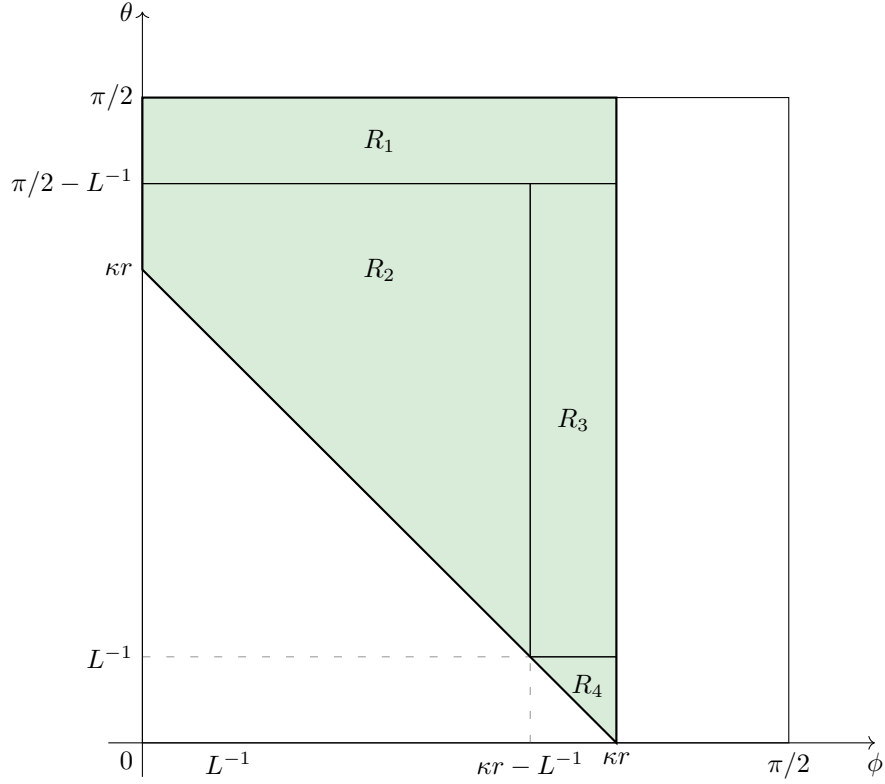
\begin{figure}[htbp]
\centering
\begin{tikzpicture}[scale=0.9]
        \def\R{0.75\textwidth} 
        \def\t{0.65\textwidth}
        \def\r{0.55\textwidth} 
        \def\s{0.45\textwidth} 
        \def\u{0.1\textwidth}
        \coordinate (A0) at (0,0);
        \coordinate (A1) at (\R,0);
        \coordinate (A2) at (0,\R);
        \coordinate (A3) at (0,\r);
        \coordinate (A4) at (\r,0);
        \coordinate (A5) at (\s,0);
        \coordinate (A6) at (0,\s);
        \coordinate (A7) at (\t,0);
        \coordinate (A8) at (0,\t);
        \coordinate (A9) at (\u,0);
        \coordinate (A10) at (0,\u);
        \coordinate (A) at (\s,\u);
        \coordinate (B) at (\r,\u);
        \coordinate (C) at (\r,0);
        \draw (A0) node[below left]{$0$};
        \draw[->] (-0.5,0)--(0.85\textwidth,0) node[below]{$\phi$};
        \draw[->] (0,-0.5)--(0,0.85\textwidth) node[left]{$\theta$};
        \draw[loosely dashed, color=gray!80] (0,\u) -- (\s,\u);
        \draw[loosely dashed, color=gray!80] (\s,0) -- (\s,\u);
        \draw (0,0) rectangle (\R,\R);
        \draw[thick, fill=Green!15] (0,\r) -- (\r,0) -- (\r,\R) -- (0,\R) -- cycle;
        \draw (A8) rectangle (\r,\R) node[pos=0.5]{$R_{1}$};
        \draw (\s,\u) rectangle (\r,\t) node[pos=0.5]{$R_{3}$};
        \draw (\s,\u) -- (\r,\u) -- (\r,0) -- cycle;
        \node at (barycentric cs:A=1,B=1,C=1) {$R_4$};
        \draw (0,\r) -- (\s,\u) -- (\s,\t) -- (0,\t) -- cycle;
        \node at (0.275\textwidth,\r) {$R_2$};
        \draw (A1) node[below]{$\pi/2$};
        \draw (A2) node[left]{$\pi/2$};
        \draw (A3) node[left]{$\kappa r$};
        \draw (A4) node[below]{$\kappa r$};
        \draw (A5) node[below]{$\kappa r-L^{-1}$};
        \draw (A8) node[left]{$\pi/2-L^{-1}$};
        \draw (A9) node[below]{$L^{-1}$};
        \draw (A10) node[left]{$L^{-1}$};
    \end{tikzpicture}
\caption{Subdivision of the region of integration in \eqref{eq:integral} when $L$ is large enough.}
\label{fig:subdivision}
\end{figure}

\begin{lemma}\label{lemma:R1}
    The following estimate holds:
    \begin{equation*}
        \iint_{R_1}\abs[\Big]{\Jacobi{L}{\alpha+1}{\beta}(\cos(2\theta))}^{2}
        (\sin\theta)^{2\alpha+1}
        (\cos\theta)^{2\beta+1}\dd\theta\dd \phi=O(L^{-2}).
    \end{equation*}
\end{lemma}

\begin{proof}
    In this region, we use $\sin\theta\leq 1$, $\cos\theta\leq \pi/2-\theta$, and
    \begin{equation*}
        \abs[\Big]{\Jacobi{L}{\alpha+1}{\beta}(\cos(2\theta))}=O(L^{\beta}).
    \end{equation*}
    Then,
    \begin{align*}
        \MoveEqLeft\iint_{R_1}\abs[\Big]{\Jacobi{L}{\alpha+1}{\beta}(\cos(2\theta))}^{2}
        (\sin\theta)^{2\alpha+1}(\cos\theta)^{2\beta+1}\dd\theta\dd \phi\\
        &=O\biggl(L^{2\beta}\int_{0}^{\kappa r} \int_{\pi/2-L^{-1}}^{\pi/2}(\pi/2-\theta)^{2\beta+1}\dd\theta\dd\phi\biggr)=O(L^{-2}).\qedhere
    \end{align*}
\end{proof}

\begin{lemma}\label{lemma:R2}
    The following estimate holds:
    \begin{equation*}
        \iint_{R_2}\abs[\Big]{\Jacobi{L}{\alpha+1}{\beta}(\cos(2\theta))}^{2}
    (\sin\theta)^{2\alpha+1}
    (\cos\theta)^{2\beta+1}\dd\theta\dd \phi=O(L^{-1}\log L).
    \end{equation*}
\end{lemma}

\begin{proof}
   We now use $\sin\theta\geq \frac{2}{\pi}\theta$ and 
    \begin{equation*}
        \abs[\Big]{\Jacobi{L}{\alpha+1}{\beta}(\cos(2\theta))}=(\sin\theta)^{-\alpha-3/2}(\cos\theta)^{-\beta-1/2}O(L^{-1/2}).
    \end{equation*}
    Then,
    \begin{align*}
        \MoveEqLeft\iint_{R_2}\abs[\Big]{\Jacobi{L}{\alpha+1}{\beta}(\cos(2\theta))}^{2}
        (\sin\theta)^{2\alpha+1}
        (\cos\theta)^{2\beta+1}\dd\theta\dd \phi\\
        &=O\biggl(L^{-1}\int_{0}^{\kappa r-L^{-1}}\int_{\kappa r-\phi}^{\pi/2-L^{-1}}\theta^{-2}\dd\theta\dd\phi\biggr)\\
        &=O\biggl(L^{-1}\int_{0}^{\kappa r-L^{-1}}\frac{1}{\kappa r-\phi}\dd\phi\biggr)\\
        &=O(L^{-1}\log L).\qedhere
    \end{align*}
\end{proof}

\begin{lemma}\label{lemma:R3}
    The following estimate holds:
    \begin{equation*}
        \iint_{R_3}\abs[\Big]{\Jacobi{L}{\alpha+1}{\beta}(\cos(2\theta))}^{2}
    (\sin\theta)^{2\alpha+1}
    (\cos\theta)^{2\beta+1}\dd\theta\dd \phi=O(L^{-1}).
    \end{equation*}
\end{lemma}

\begin{proof}
    Again, we use $\sin\theta\geq \frac2\pi\theta$, and
     \begin{equation*}
        \abs[\Big]{\Jacobi{L}{\alpha+1}{\beta}(\cos(2\theta))}=(\sin\theta)^{-\alpha-3/2}(\cos\theta)^{-\beta-1/2}O(L^{-1/2}).
    \end{equation*}
    Then,
    \begin{align*}
        \MoveEqLeft \iint_{R_3}\abs[\Big]{\Jacobi{L}{\alpha+1}{\beta}(\cos(2\theta))}^{2}
        (\sin\theta)^{2\alpha+1}
        (\cos\theta)^{2\beta+1}\dd\theta\dd \phi\\
        &=O\biggl(L^{-1}\int_{\kappa r-L^{-1}}^{\kappa r}\int_{L^{-1}}^{\pi/2-L^{-1}}\theta^{-2} \dd\theta\dd\phi\biggr)=O(L^{-1}),
    \end{align*}
    by direct computation.
\end{proof}

\begin{lemma}\label{lemma:R4}
    The following estimate holds:
    \begin{equation*}
        \iint_{R_4}\abs[\Big]{\Jacobi{L}{\alpha+1}{\beta}(\cos(2\theta))}^{2}
    (\sin\theta)^{2\alpha+1}
    (\cos\theta)^{2\beta+1}\dd\theta\dd \phi=O(L^{-1}).
    \end{equation*}
\end{lemma}

\begin{proof}
    In this region, we use $\cos\theta\leq 1$, $\sin\theta\leq \theta$, and
    \begin{equation*}
        \abs[\Big]{\Jacobi{L}{\alpha+1}{\beta}(\cos(2\theta))}=O(L^{\alpha+1}).
    \end{equation*}
    Then,
    \begin{align*}
        \MoveEqLeft\iint_{R_4}\abs[\bigg]{\Jacobi{L}{\alpha+1}{\beta}(\cos(2\theta))}^{2}
        (\sin\theta)^{2\alpha+1}
        (\cos\theta)^{2\beta+1}\dd\theta\dd \phi\\
        &=O\biggl(L^{2\alpha+2}\int_{\kappa r-L^{-1}}^{\kappa r}\int_{\kappa r -\phi}^{L^{-1}}\theta^{2\alpha+1}\dd\theta\dd\phi\biggr)\\
        &=O\biggl(L^{2\alpha+2}\int_{\kappa r-L^{-1}}^{\kappa r}\bigl(L^{-2\alpha-2}-(\kappa r-\phi)^{2\alpha+2}\bigr)\dd\phi\biggr)\\
        &=O(L^{-1}).\qedhere
    \end{align*}
\end{proof}

From the previous lemmas, it follows that $I=O(L^{-1}\log L)$, which completes the proof.

\section{The variance of the projective ensemble on $\CP^d$.\\ Proof of \cref{th:Variancep}}\label{sec:variancesproj}

In this section, we prove \cref{th:Variancep}. The proof is very similar to that of \cref{th:Variance-harmonic}, with the kernel replaced by that of the projective ensemble. In this case, using the notation of \cref{sec:variancesh}, the integral to be estimated becomes
\begin{equation*}
    \Variance(N_A)\leq N^2\int_{p\in A}\int_{q\in B(p,r-s)^c}\abs[\bigg]{\innerprod[\bigg]{\frac{p}{\norm{p}}}{\frac{q}{\norm{q}}}}^{2L}\dd q \dd p.
\end{equation*}
Note that (the absolute value of) the inner product inside the integral is precisely the cosine of the Riemannian distance from $p$ to $q$. Again, proceeding as in \cref{sec:variancesh}, we can write
\begin{equation*}
    \Variance(N_A)\leq 4d^2N^2\int_{0}^r(\sin\phi)^{2d-1}\cos(\phi)\int_{r-\phi}^{\pi/2}(\sin\theta)^{2d-1}\cos(\theta)\cos^{2L}\theta\dd\theta\dd\phi.
\end{equation*}
We bound this integral by elementary means, using that $\cos\phi,\cos\theta\leq 1$ and that
\begin{equation*}
    \cos^{2}\theta\leq e^{-\theta^2/2},\quad \theta\in[0,\pi/2],
\end{equation*}
which implies
\begin{equation*}
    \Variance(N_A)\leq4d^2N^2\int_{0}^r\phi^{2d-1}\int_{r-\phi}^{\pi/2}\theta^{2d-1}e^{-\theta^2L/2}\dd\theta\dd\phi.
\end{equation*}
The change of variables $\theta^2L/2=v$ allows us to bound the inner integral by
\begin{equation*}
    \frac{1}{L^d}\int_{L(r-\phi)^2/2}^\infty (2v)^{d-1}e^{-v}\dd v.
\end{equation*}
Then, using \eqref{eq:dyL} we obtain, for some constant $K>0$ depending only on $d$,
\begin{equation*}
    \frac{\Variance(N_A)}{L^{d-1}}\leq K L\int_0^r\phi^{2d-1}\int_{L(r-\phi)^2/2}^\infty v^{d-1}e^{-v}\dd v\dd\phi.
\end{equation*}
A further change of variables $v=t+L(r-\phi)^2/2$ in the inner integral gives
\begin{equation*}
    \frac{\Variance(N_A)}{L^{d-1}}\leq K L\int_0^r\frac{\phi^{2d-1}}{e^{L(r-\phi)^2/2}}\int_{0}^\infty (t+L(r-\phi)^2/2)^{d-1}e^{-t}\dd t\dd\phi.
\end{equation*}
Note that
\begin{align*}
    \int_{0}^\infty (t+L(r-\phi)^2/2)^{d-1}e^{-t}\dd t
    &=\sum_{j=0}^{d-1}\binom{d-1}{j}\int_{0}^\infty t^j\Bigl(L(r-\phi)^2/2\Bigr)^{d-1-j}e^{-t}\dd t\\
    &=\sum_{j=0}^{d-1}\binom{d-1}{j}\Gamma(j+1)\Bigl(L(r-\phi)^2/2\Bigr)^{d-1-j},
\end{align*}
from which
\begin{align*}
    \frac{\Variance(N_A)}{L^{d-1}}&\leq K\sum_{j=0}^{d-1} L\int_0^r\frac{\phi^{2d-1}}{e^{L(r-\phi)^2/2}}\Bigl(L(r-\phi)^2/2\Bigr)^{d-1-j}\dd \phi\\
    &=K\sum_{j=0}^{d-1} \int_0^{Lr^2/2}\frac{\bigl(r-\sqrt{2u/L}\bigr)^{2d-1}}{e^{u}}u^{d-1-j}\frac{\sqrt{L}}{\sqrt{2u}}\dd u\\
    &\leq K\sum_{j=0}^{d-1} \int_0^{\infty}\frac{1}{e^{u}}u^{d-1-j}\frac{\sqrt{L}}{\sqrt{2u}}\dd u\\
    &\leq K \sqrt{L}.
\end{align*}
We thus have
\begin{equation*}
    \Variance(N_A)\leq KL^{d-1/2},
\end{equation*}
and \cref{th:Variancep} follows from \eqref{eq:dyL}.


\bibliographystyle{amsplain}

\end{document}